\documentclass[a4paper,12pt,reqno]{amsart}

\usepackage{macros} 
\usepackage{bbm}

\upsymbols{Tor,Isom,Div,Hom,inv,WC,nr,red,GL,SL,Aut,tors,Hom,Inv,Prob}
\upsymbols{Gal,res,im,coker,Frob,rk,Res,Ext,disc,Id,FS,cyc,PGL,cor,Vol,loc,Reg,Cond}
\bbsymbols{A,G,Z,Q,R,F,C,N,E,P}
\scrsymbols{F,C,S,L,M}
\calsymbols{O,T,F,G,C,D,A,M,I,N}
\fraksymbols{p,P,a,b,d,c,m}
\newcommand{\one}{\mathbbm{1}}
\usepackage[margin=1.2in]{geometry}

\newcommand{\sel}[1]{\textup{Sel}_{#1}}

\DeclareMathOperator{\rad}{rad}
\DeclareMathOperator{\re}{Re}
\DeclareMathOperator{\Brun}{Br_{un}}
\DeclareMathOperator{\Br}{Br}
\let \L \relax
\DeclareMathOperator{\L}{L}

\newcommand{\N}{\mathbb{N}}
\newcommand{\Q}{\mathbb{Q}}

\newcommand{\Z}{\mathbb{Z}}

\numberwithin{equation}{section}

\title{Lower bounds for counting \texorpdfstring{$A_4$}{A4}-quartic fields}
\author[D. Loughran]{Daniel Loughran}
\address{Department of Mathematical Sciences \\
University of Bath \\
Claverton Down \
Bath\\ 
BA2 7AY\\
UK.}
\urladdr{https://sites.google.com/site/danielloughran/}

\author[R. Paterson]{Ross Paterson}
\address{School of Mathematics, School of Mathematics, University of Bristol, Bristol, BS8 1TW, UK, and the Heilbronn Institute for Mathematical Research, Bristol, UK}
\email{rosspatersonmath@gmail.com}
\urladdr{https://ramifiedprime.github.io}

\makeatletter
\hypersetup{
pdftitle={\@title},
pdfauthor={Daniel Loughran, Ross Paterson},
pdfsubject={Number Theory, Malle's conjecture, number fields, stacks},
pdfkeywords={}
}
\makeatother

\begin{document}

\begin{abstract}
    A conjecture of Malle predicts the quantity of number fields with bounded discriminant of given Galois group. We present a lower bound matching this in the case of quartic fields with Galois group $A_4$.
\end{abstract}

\maketitle

\tableofcontents

\section{Introduction}

For a transitive permutation group $G \subseteq S_n$, a conjecture of Malle \cites{Mal02,Mal04} predicts the asymptotic distribution of degree $n$ number fields of bounded discriminant with Galois group $G$. This conjecture has been resolved completely for $G$ abelian \cite{MR969545}, and for numerous groups of small degree, including for example $G= D_4,S_4,S_5$ \cites{Dihedral, Bha05,Bha10}. However a particularly challenging open case is that of quartic fields with Galois group $A_4$, where Malle's conjecture predicts an asymptotic formula of the shape $X^{1/2} \log X$. Our main result is a correct lower bound for this counting problem.

\begin{theorem} \label{thm:main}  For sufficiently large $X$, we have
    \[\#\set{L/\QQ~:~
    \begin{array}{l}
        L/\QQ\textnormal{ is an }A_4\textnormal{-quartic with } \Delta_L\leq X
    \end{array}}
    \gg X^{1/2} \log X.\]
\end{theorem}

The previous best lower bound  $X^{1/2}$ is from 1980 due to Baily \cite{Bai80}*{Thm.~3}.  As for upper bounds, the current record is \cite{BSTTTZ}*{Thm.~1.4}, which gives approximately $X^{0.7785}$. We note that the current best lower bounds for $A_n$ for $n \geq 5$ come nowhere close to the expected size \cite{LLOT}.

Baily achieved his lower bound by fixing a cyclic cubic field $F$ and counting counting $A_4$-quartic fields whose cubic resolvent is equal to $F$. This perspective was taken up further in \cite{MR3215550}, where the Dirichlet series of such fields is calculated using a parametrisation of quartic fields with given cubic resolvent. A key step in our proof is the following lower bound for quartic fields with a given cubic resolvent with, crucially, an effective error term.

\begin{theorem} \label{thm:fixed_cubic} Let $F/\QQ$ be a cyclic cubic field, and $\epsilon>0$. Then for $X\geq 64\Delta_F^{1+1/\epsilon}$ we have
\begin{align*}
    &\#\set{L/\QQ~:~
    \begin{array}{l}
    L/\QQ\textnormal{ is an }A_4\textnormal{-quartic with}\\
    \textnormal{ cubic resolvent }F\textnormal{ and }\Delta_L\leq X
    \end{array} } 
    \\&\geq \frac{\zeta^*_F(1)}{6\Delta_F^{1/2}}c_F X^{1/2} + O_\epsilon\braces{\#\Cl_F[2] \cdot 2^{3\omega(\Delta_F)}\left(\frac{X}{\Delta_F}\right)^{1/4+\epsilon} \log(X)^{79}}.
\end{align*}
    where the implied constant is independent of $F$ and 
    \[c_F=\prod_{p \text{ ramified in } F}\left(1 - \frac{1}{p}\right) \prod_{p \text{ split in } F}\left(1 - \frac{1}{p}\right)^3\left(1 + \frac{3}{p}\right) \prod_{p \text{ inert in } F}\left(1 - \frac{1}{p^3}\right)
    \]
    and $p$ varies over all primes.
\end{theorem}

In the statement $\zeta^*_F(1)$ denotes the residue of the Dedekind zeta function $\zeta_F$ of $F$ at $s =1$. Our method of proof can be adapted to yield an asymptotic formula, but with a more complicated local Euler factor at $2$.

The next step is to sum over all $F$ with discriminant at most $X$. An obvious challenge is controlling the class group factor in the error term of Theorem~\ref{thm:fixed_cubic}. We tackle this as follows: The expected main term is $X^{1/2}\log X$, so the contribution from the sum over $F$ with discriminant up to $X$ should be $\log X$, and so summing over the subset of $F$ with discriminant up to $X^\epsilon$ should still give the correct lower bound.  In this reduced range the Brauer--Siegel theorem implies that the class group is also $\ll X^\varepsilon$, so will not cause problems in the error term when $\epsilon$ is sufficiently small.  We expect this trick to be applicable for controlling error terms in other lower bounds in Malle's conjecture.

To take care over the sum of leading constants, we need to sum over the residues $\zeta_F^*(1)$. From the perspective of the class number formula, these contain rather complicated arithmetic information which is hard to control pointwise. For example, the best \cite{Lou00}*{Thm.~1} pointwise upper bounds for general number fields $k$ of fixed degree has the shape $O(\log \left(\abs{\Delta_k}\right)^{\deg k - 1})$; in fact GRH only gives $O(\log \log \left(\abs{\Delta_k}\right)^{\deg k})$ \cite{MR3682635}*{(1.1)}. Nonetheless we show that $\zeta_F^*(1)$ has constant average order.

\begin{theorem}\label{thm:LTAvg}
    For every $C>0$ we have
    \[\sum_{\substack{F\textnormal{ cyclic cubic}\\\textnormal{unram. at }3\\\Delta_F\leq X}}\zeta_F^*(1)=\frac{7\pi}{26\sqrt{3}}\alpha X^{1/2} + O_C(X^{1/2}\log(X)^{-C}),\]
    where 
    \[\alpha=\prod_{p\equiv 1\bmod 3} \braces{1-\frac{1}{p}}^2\braces{1 +\frac{2}{p} + \frac{1 +3p^{-1} + p^{-2} - p^{-6}}
    {p^2(1 - p^{-3})^2}}
    \prod_{p\equiv 2\bmod 3}\braces{\frac{1+p^{-3}}{1-p^{-3}}}.\]
\end{theorem}

Theorem \ref{thm:LTAvg} may look classical, but we could only find the average zeta residue for quadratic extensions (see e.g.~\cite{Jut73}*{Thm.~3} and \cite{Wolke}*{Thm.~1}). In our setting, the Brauer--Siegel Theorem \cite{MR1282723}*{Ch.~XVI} states that $\log(h_F R_F)/ \log (\Delta_F^{1/2}) \to 1$ as $\Delta_F \to \infty$. Thus via the class number formula, Theorem \ref{thm:LTAvg} may be viewed as an averaged version of the Brauer--Siegel Theorem without the logarithmic weighting. We actually require a slightly more general version of Theorem \ref{thm:LTAvg} which gives control of ramification in order to input the multiplicative function $\prod_{p \mid \Delta_F}(1-1/p)$ from Theorem \ref{thm:fixed_cubic} (see Theorem \ref{thm:LTAvg_d}).

We prove Theorem \ref{thm:LTAvg} by expanding $\zeta_F$ into a product of Dirichlet $L$-functions and reduce to obtaining an asymptotic formula for the second moment $|L(1,\chi)|^2$ as $\chi$ varies over cubic Dirichlet characters of bounded conductor.  This is then achieved by character sum techniques, applying a version of the bilinear sieve over number fields \cite{MR4564991}*{Prop.~4.21} and a Siegel--Walfisz-type result over number fields from \cite{KPSS}*{Lem.~4.8}.  We expect results like Theorem \ref{thm:LTAvg} to play a similar role in other inductive approaches to Malle's conjecture, since special values of Artin L-functions appear in the leading constant for Malle's conjecture as part of the Tamagawa measure \cite{LS24}*{(8.4)}, and to be achievable in the abelian case by a similar strategy.

An inductive perspective for Malle's conjecture is discussed in \cite{LS24}*{Conj.~9.6} and \cite{ALOWW25}. However the key difference to our setting is that this framework is for \emph{unbalanced} (or \emph{concentrated}) heights, which means that the collection of minimal index elements do not generate $G$. In the unbalanced setting, the expectation is that one should sort according to the fibration given by quotienting out by the normal subgroup generated by the minimal weight elements (the Iitaka fibration), and that the sum of the leading constants converges. However our situation corresponds to a \textit{balanced height}, and the sum of the leading constants over our fibration does not converge. In fact it is controlling the sum over these leading constants which is a challenge in our proof, with \Cref{thm:LTAvg} being the key ingredient. 

\subsection{Relation to stacks and Manin's conjecture}  \label{sec:stacks_intro}
Malle's conjecture has recently been reinterpreted in a stacky framework by Ellenberg--Satriano--Zurieck-Brown \cite{MR4557890} and Darda--Yasuda \cites{MR4639951,DardaYasudaDMStacks}, with a prediction of Loughran--Santens \cite{LS24} regarding the leading constant for Malle's conjecture. Our approach can be put into this framework as follows: We count rational points on the classifying stack $BA_4$ by partitioning according to the fibration $BA_4\to B\ZZ/3\ZZ$ induced by a choice of surjective homomorphism $A_4\to \ZZ/3\ZZ$;  this map associates to an $A_4$-quartic its cubic resolvent. The fibres of this map are classifying stacks of a group scheme with underlying abelian group the Klein $4$-group. Therefore the counts in Theorem \ref{thm:fixed_cubic} fit into the version of Malle's conjecture for group schemes. We verify in \Cref{sec:stacks} that the leading constant obtained in Theorem \ref{thm:fixed_cubic} agrees with the conjectural leading constant from \cite{LS24}*{Conj.~9.1}, up to $2$-adic and archimedean factors.

In fact, the innovation of restricting to cyclic cubic fields with small discriminant to control error terms is inspired by a similar trick for lower bounds for Manin's conjecture for conic bundle surfaces utilised in \cite{FLS18}*{\S4.3}. Here one is given a conic bundle surface $X \to \PP^1$ and sorts rational points of bounded height on $X$ according to the conic fibration; one controls the errors terms from the count on each conic by restricting to conics of small height, in an analogous way to our paper.

Moreover, in recent work \cite{BDLPR} on Manin's conjecture for the symmetric square of a Fano variety, a key ingredient \cite{BDLPR}*{Prop.~1.8} is an asymptotic for moments of special values of Artin $L$-functions in quadratic twist families, in a similar manner to our Theorem \ref{thm:LTAvg} on cyclic cubic families. These parallel developments highlight the common tool kit required for both Manin's and Malle's~Conjecture.

\subsection{Overview of the paper}
In \S \ref{sec:squarefree} we study problems related to counting ideals of a number field with squarefree norm, which will arise in various parts of our proofs. Our main result here is Theorem \ref{thm:idealcounting}, which gives a count for such ideals with class group conditions imposed and with an effective error term.

In \S \ref{sec:cubic} we prove Theorem \ref{thm:LTAvg} and in \S \ref{sec:quartic} we prove Theorems \ref{thm:fixed_cubic} and \ref{thm:main}. We finish in \S \ref{sec:stacks} by explaining how Theorem \ref{thm:fixed_cubic} agrees with the predictions from \cite{LS24} for counting points on the stack $BA_4$.

\subsection{Notation and Conventions}
For each number field $F$, we write $\cO_F$ for the ring of integers, $I_F$ for the set of ideals of $\cO_F$, $\Delta_F$ for the (absolute) discriminant, and $\Cl_F$ for the ideal class group.  We denote by $\zeta_F(s)$ the Dedekind zeta function of $F$, and by $\zeta_F^*(1)$ its residue at $1$.  For each ideal $J\leq \cO_F$, we write $\rad(J)$ for the radical (i.e. the squarefree product of prime ideals dividing $J$), and $\cN(J)$ the norm. We let $\zeta_3 = e^{2 \pi i/3}$ and denote by $\art{\cdot}{\cdot}_3$ a fixed choice of cubic residue symbol.

\subsection*{Acknowledgements}
Daniel Loughran
was supported by UKRI Future Leaders Fellowship
\texttt{MR/V021362/1}.
Ross Paterson was supported by the Heilbronn Institute for Mathematical Research.

\section{Ideals of squarefree norm} \label{sec:squarefree}
Throughout the section, we shall be concerned with the problem of counting ideals of squarefree norm which is coprime to $M$, together with certain ideal class constraints, which will be useful in what follows.  

\subsection{\texorpdfstring{$\mathbb{Q}$}{Q}-multiplicative functions}
In what follows we will need to deal with instances of functions on the monoid of ideals $I_F$ of a number field $F$.  The functions of interest are given by taking the norm of an ideal, then applying a multiplicative function on $\ZZ$.  Such a function need not be a multiplicative function on $I_F$. However they satisfy a weak version of multiplicativity. We study briefly the formal properties of such functions here. We skip proofs which are completely elementary.

\begin{definition}
    Say that a function $f:I_F\to\CC$ is \textit{$\QQ$-multiplicative} if for every pair $\fa,\fb\in I_F$ such that $\cN(\fa)$ is coprime to $\cN(\fb)$ we have
    \[f(\fa\fb)=f(\fa)f(\fb).\]
    We say that $f$ is \textit{multiplicative} if the same holds when $\fa$ and $\fb$ are coprime ideals.
\end{definition}
\begin{definition}
    Let $\mu_F:I_F\to \set{\pm 1}$ be the M\"{o}bius  function defined on powers of prime ideals $\fp^k$ by
    \[\mu_F(\fp^k):=\begin{cases}-1&\textnormal{if }k=1,\\0&\textnormal{if }k\geq 2.\end{cases}\]
\end{definition}

\begin{definition}
    The Dirichlet convolution of two functions $f,g:I_F\to \CC$ is $f\star g:I_F\to\CC$ defined by
    \[f\star g(\fa):=\sum_{\fb\mid\fa}f(\fb)g(\fa\fb^{-1}).\]
\end{definition}

\begin{lemma}\label{lem:inverseconv}
    Let $f:I_F\to \CC$ be any function.  Then
    \[f(\fa)=\sum_{\fb\mid\fa}(f\star\mu_F)(\fb)\]
\end{lemma}

\begin{lemma}\label{lem:convolutionofKmults}
    If $f,g:I_F\to\CC$ are both $\QQ$-multiplicative then so is $f\star g$.
\end{lemma}

\subsection{A special \texorpdfstring{$\mathbb{Q}$}{Q}-multiplicative function}

Let $\one_M$ denote the indicator function on $\Z$ which takes value $1$ on integers coprime to $M$ and $0$ otherwise. The indicator function for squarefree ideals of a number field $F$, whose norm is coprime to $M$ is the $\Q$-multiplicative function defined by $\fa\mapsto \mu^2(\cN(\fa))\one_{M}(\cN(\fa))$.   We now consider its Dirichlet inverse, which will play a key role in what follows.

\begin{definition}\label{def:gamma}
    For each number field $F$ and integer $M$, define the $\QQ$-multiplicative function 
    \[\gamma_{F,M}:=\braces{\mu^2(\cN(\cdot))\one_{M}(\cN(\cdot))}\star\mu_F.\]
    When $M=1$ then we will simply write $\gamma_F:=\gamma_{F,1}$.
\end{definition}

We want a more concrete description, so that we can bound this function.  By \Cref{lem:convolutionofKmults} the function $\gamma_{F,M}$ is $\QQ$-multiplicative, so it suffices to determine the behaviour above each fixed rational prime $p$.

\begin{lemma}\label{lem:gammaexplicit}
Let $F/\QQ$ be a finite Galois extension, $\set{\fp_i~:~i=1,\dots,t}$ be a set of distinct prime ideals of $\cO_F$ with equal residue characteristic $p$, and let $f_p$ denote the inertia degree at $p$ in $F/\QQ$.  If $p \nmid M$, then for each vector $\mathbf{a}=(a_i)_i\in \ZZ_{\geq 1}^t$ we have
\[\gamma_{F,M}\braces{\prod_{i=1}^t\fp_i^{a_i}}=\begin{cases}
    (-1)^t, & \textnormal{if }f_p>1\textnormal{ and }\mathbf{a}=(1,\dots,1);\\
    (-1)^{t-1}(t-1), &\textnormal{if } f_p=1\textnormal{ and }\mathbf{a}=(1,\dots,1);\\
    (-1)^t, & \textnormal{if }f_p=1\textnormal{ and }\mathbf{a}=(1,\dots,1) + \mathbf{e}_k\textnormal{ for some }k;\\
    0, &\textnormal{else,}
\end{cases}\]
where $\mathbf{e}_k$ denotes the elementary basis vector with the value $1$ at the $k$th coordinate and $0$ at all others.
Else, if $p\mid M$ then 
\[\gamma_{F,M}\braces{\prod_{i=1}^t\fp_i^{a_i}}=\begin{cases}
    (-1)^t&\textnormal{if }\mathbf{a}=(1,\dots,1);\\
    0&\textnormal{else}
\end{cases}\]
\end{lemma}
\begin{proof}
    Note that by definition
    \[\gamma_{F,M}\braces{\prod_{i=1}^t\fp_i^{a_i}}=\sum_{\mathbf{b}\in\set{0,1}^t}\mu_F\braces{\prod_i\fp_i^{b_i}}\mu^2\braces{p^{f_p\sum_ia_i-b_i}}\one_M\braces{p^{f_p\sum_ia_i-b_i}},\]
    from which the claim follows by simple case analysis.
\end{proof}

\begin{lemma}\label{lem:gammabound}
  Let $d\geq 1$ be an integer.  For every integer $M\geq 1$, every finite Galois extension $F/\QQ$ of degree $d$, and every real number $X>1$ we have
    \[\sum_{\substack{\fb\leq \cO_F\\\cN(\fb)\leq X}}\abs{\gamma_{F,M}(\fb)}\ll_d 2^{d\omega(M)}X^{1/2}\log(X)^{2^{d}d^2-1},\]
    where the implied constant depends only on the degree $d$, and is uniform in the choice of $M$, $X$ and $F$. 
    In particular, the Dirichlet series 
    $L(\gamma_{F,M},s)=\sum_{\substack{\fa\leq \cO_F}}\gamma_{F,M}(\fa)\cN(\fa)^{-s}$ converges absolutely for $\re s > 1/2$. 
\end{lemma}
\begin{proof}
    Firstly, by definition we can rewrite each ideal $\fb$ as $\fb=\fb'\fb''$ where $\fb'$ and $\fb''$ are coprime and $\fb'$ is supported on primes $\fp\mid M$.  Since $\gamma_{F,M}$ is $\QQ$-multiplicative, we break the sum across this factorisation and then estimate as follows
    \begin{align*}
    &\sum_{\substack{\fb\leq \cO_F\\\cN(\fb)\leq X}}\abs{\gamma_{F,M}(\fb)}
    \\&=\sum_{\substack{\fb'\leq \cO_F \textnormal{ supported }\\\textnormal{on primes }\fp\mid M}}\abs{\gamma_{F,M}(\fb')}\sum_{\substack{\fb''\leq \cO_F\\\textnormal{coprime to $M$}\\\cN(\fb'')\leq \frac{X}{\cN(\fb')}}}\abs{\gamma_{F}(\fb'')}
    \\&\leq \#\set{\fb'\leq \cO_F~:~
                    \fb'\textnormal{ squarefree and divides }M}
            \sum_{\substack{\fb''\leq \cO_F\\\cN(\fb'')\leq X}}\abs{\gamma_F(\fb'')},
    \end{align*}
    where the final inequality follows from \Cref{lem:gammaexplicit}.
    We then bound 
    \[\#\set{\fb'\leq \cO_F~:~
                    \fb'\textnormal{ squarefree and divides }M}
    \leq 2^{d\omega(M)}.
    \]
    Hence it remains to prove the claimed bound when $M=1$.  If $\gamma_F(\fb)\neq 0$ then it follows from \Cref{lem:gammaexplicit} that $\rad(\cN(\fb))^2\mid \cN(\fb)$ and so when $\cN(\fb)\leq X$ we must have $\rad(\cN(\fb))\leq X^{1/2}$. This yields
    \begin{align*}
        \sum_{\substack{\fb\leq \cO_F\\\cN(\fb)\leq X}}\abs{\gamma_{F}(\fb)}
        =\sum_{\substack{b\leq X^{1/2}\\\text{squarefree}}}\sum_{\substack{\fb\leq \cO_F\\\cN(\fb)\leq X\\\rad(\cN(\fb))=b}} \abs{\gamma_F(\fb)}.
    \end{align*}
    Now, for fixed squarefree $b$ we wish to estimate $\abs{\gamma_F(\fb)}$ for each $\fb$ with $\rad(\cN(\fb))=b$.  Such an ideal $\fb$ must be of the form 
    
    \[\fb=\prod_{p\mid b}\prod_{\fp\mid p}\fp^{a_\fp},\]
    where the outer product is over rational primes and the inner over primes of $F$.  Moreover if we insist that $\gamma_F(\fb)\neq 0$, then by \Cref{lem:gammaexplicit}, for each rational $p\mid b$ there are at most 
    \[\sum_{t=1}^{d}\binom{d}{t}(t+1)=d2^{d-1}+2^{d}-1 = 2^{d}\braces{\frac{d}{2}+1}-1\leq 2^dd\]
    choices of the exponents $(a_\fp)_{\fp\mid p}$ such that $\gamma_F(\prod_{\fp\mid p}\fp^{a_\fp})\neq 0$.  Hence, for each squarefree $b$ there are at most $\braces{2^{d}d}^{\omega(b)}$ choices of $\fb$ such that $\gamma_F(\fb)\neq 0$ and $\rad(\cN(\fb))=b$.    Further, it follows from \Cref{lem:gammaexplicit} that we can bound $\abs{\gamma_F(\fb)}\leq d^{\omega(b)}$.

    Putting these together, we have
    \[\sum_{\substack{\fb\leq \cO_F\\\cN(\fb)\leq X}}\abs{\gamma_F(\fb)}\leq 
    \sum_{\substack{b\leq X^{1/2}\\\text{squarefree}}}
    \braces{2^dd^2}^{\omega(b)}.\]
    We now apply \cite{MR552470}*{Theorem 1} to bound the sum of these uniformly by 
    the required $X^{1/2}\log(X)^{2^dd^2-1}$.
\end{proof}

We record here some special values of the Dirichlet series $L(\gamma_{F,M},s)$ which will be used later. (Compare Lemma \ref{lem:L_gamma_F} with Theorem \ref{thm:fixed_cubic}).
\begin{lemma} \label{lem:L_gamma_F}
    Let $F$ be a cyclic cubic field. Then
    $$L(\gamma_{F,\Delta_F},1) = \prod_{p\text{ ramified in }F}\braces{1 - \frac{1}{p}}
    \prod_{p \text{ split in } F}\left(1 - \frac{1}{p}\right)^3\left(1 + \frac{3}{p}\right) 
    \prod_{p \text{ inert in } F}\left(1 - \frac{1}{p^3}\right).$$
\end{lemma}
\begin{proof}
    As $\gamma_{F,\Delta_F}$ is $\Q$-multiplicative, we have 
    $$L(\gamma_{F,\Delta_F},1) = \prod_p\left(\sum_{\substack{\mathfrak{a} \\ \rad(\cN(\mathfrak{a})) \mid p}}  \frac{\gamma_F(\mathfrak{a})}{\cN(\mathfrak{a})}\right).$$
    Let $p$ be a prime number. We calculate the factor at $p$ using Lemma \ref{lem:gammaexplicit}.

    First assume that $p$ is ramified in $F$, so that $p\mid \Delta_F$. Then there is a unique prime of $F$ above $p$ and it has ramification degree $3$. Then Lemma \ref{lem:gammaexplicit} implies that we obtain the local factor $1 - 1/p$.

    If $p$ is inert in $F$ then by Lemma \ref{lem:gammaexplicit} the only non-zero values occur at the trivial ideal and the unique prime above $p$. In which case we obtain the local factor $1 - 1/p^3$.

    Now assume that $p$ is completely split in $F$. The non-trivial ideals which contribute are $\fp_1^{a_1}\fp_2^{a_2}\fp_3^{a_3}$ where $(p) = \fp_1\fp_2\fp_3$ and $(a_1,a_2,a_3)$ being a permutation of $(0,1,1), (1,1,1), (0,0,2), (0,1,2)$ or $(1,1,2)$. This yields
    $$ 1 - \frac{3}{p^2} + \frac{2}{p^3} - \frac{3}{p^2} + \frac{6}{p^3} - \frac{3}{p^4} 
=  1 - \frac{6}{p^2} + \frac{8}{p^3} - \frac{3}{p^4} = \left(1 - \frac{1}{p}\right)^3\left(1 + \frac{3}{p}\right) $$
    as required.
\end{proof}

\begin{lemma} \label{lem:L_gamma_3}
    Let $M \in \N$ with $3 \mid M$. Then
    $$L(\gamma_{\Q(\zeta_3),M},1) = 
    \frac{2}{3}\prod_{\substack{p\equiv 1 \bmod 3 \\ p \nmid M}}\left( 1 - \frac{1}{p}\right)^2\left(1 + \frac{2}{p}\right) \prod_{\substack{p\equiv 1 \bmod 3\\p\mid M}}\left( 1 - \frac{1}{p}\right)^2 \prod_{p\equiv 2 \bmod 3}\braces{1-\frac{1}{p^2}}.$$
\end{lemma}
\begin{proof}
    Let $p$ be a rational prime with $p \nmid M$. If $p \equiv 2 \bmod 3$ then $p$ is inert and then
    by Lemma~\ref{lem:gammaexplicit} the only non-zero value occurs at the trivial ideal and the unique prime above $p$. In which case we obtain the local factor $1 - 1/p^2$.  On the other hand, if $p \equiv 1 \bmod 3$ then the non-trivial ideals which contribute are $\fp_1^{a_1}\fp_2^{a_2}$ where $(p) = \fp_1\fp_2$ and $(a_1,a_2)$ being a permutation of $(1,1), (0,2), (1,2)$. This yields
    $$ 1 - \frac{1}{p^2} - \frac{2}{p^2} + \frac{2}{p^3}= 1 - \frac{3}{p^2} + \frac{2}{p^3}
    = \left(1 - \frac{1}{p}\right)^2\left(1 + \frac{2}{p}\right).$$
    Now assume that $p \mid M$.
    For $p=3$ there is a unique prime ideal over $3$, and we obtain
    $1 - 1/3= 2/3$.
    For $p \equiv 1 \bmod 3$ there are two primes ideals and we obtain 
    $1 - 2/p + 1/p^2$. For $p \equiv 2 \bmod 3$ there is a unique prime ideal and we obtain 
    $1-1/p^2$.
\end{proof}

\subsection{Counting ideals}

We now arrive at our target problem for this section:  we wish to count ideals with squarefree norm in a fixed number field with an effective error term.  We will make use of the following lemma, which is a simplification of a result of Loury-Douda--Taniguchi--Thorne.
\begin{lemma}[\cite{MR4381213}*{Theorem 3}]\label{lem:LDTT}
    Let $d\in\ZZ_{>0}$.  For every $X\geq1$, and every number field $F$ of degree $d$ such that $\abs{\Delta_F}\leq X$, 
\[\#\set{\mathfrak{a}\leq \mathcal{O}_F~:~ \mathcal{N}(\fa)\leq X}=\zeta_F^*(1) X + O_{d}\braces{X^{\frac{d}{d+1}}\log(X)^{d-1}},\]
where the implied constant depends only on the degree $d$.
\end{lemma}
\begin{proof}
    Their result states that
    \[\#\set{\mathfrak{a}\leq \mathcal{O}_F~:~ \mathcal{N}(\fa)\leq X}=\zeta_F^*(1) X + O_{d}\braces{\abs{\Delta_F}^{\frac{1}{d+1}}X^{\frac{d-1}{d+1}}\log(X)^{d-1}},\]
    provided that the error term is bounded by the main term.  Now, by the Brauer--Siegel theorem we know that for $\epsilon>0$ we have $\zeta_F^*(1)\gg_{\epsilon} \abs{\Delta_F}^{-\epsilon}$.  Hence the main term is at least $X^{1-\epsilon}$, whereas the error is at most $X^{\frac{d}{d+1}}\log(X)^{d-1}$.  Choosing $\epsilon<1/(d+1)$ then provides the required bound.
\end{proof}

\begin{proposition}\label{prop:sqfideals}  Let $d\in\ZZ_{>0}$ and $\epsilon>0$.  For every number field $F$ of degree $d$, integer $M\geq 1$, and real number $X\geq\abs{\Delta_F}^{1/\epsilon}$ we have
    \[\sum_{\substack{\fa\leq \cO_F \\\cN(\fa)\leq X}} \mu^2(\cN(\fa))\one_M(\cN(\fa))
    =\zeta^*_F(1)L(\gamma_{F,M},1)X+O_d\braces{2^{d\omega(M)}X^{1/2+\epsilon}\log(X)^{2^d(d^2+1)-1}}\]
    where the implied constant depends only on the degree $d$, and in particular is independent of $F$, $X$, and $M$.
\end{proposition}
\begin{proof}
    We use the identity $\mu^2(\cN(\fa))\one_M(\cN(\fa))=\sum_{\fb\mid\fa}\gamma_{F,M}(\fb)$ (see \Cref{lem:inverseconv}) to obtain that our sum equals
\begin{equation}\label{eq:prop_counting_ideals}
\sum_{\substack{\fa\leq \cO_F \\
\cN(\fa)\leq X}} \hspace{-5pt} \mu^2(\cN(\fa))\one_M(\cN(\fa))
= \sum_{\substack{\fb \leq \cO_F \\ \cN(\fb) \leq X}} \hspace{-5pt} \gamma_{F,M}(\fb)
\#\set{\fa\leq \cO_F~:~\cN(\fa)\leq \frac{X}{\cN(\fb)}}.
\end{equation}
Let $B:=X^{1-\epsilon}$.  We will split our summation into the regions $\mathcal{N}(\fb)\leq B$ and $\mathcal{N}(\fb)>B$.  Firstly, we consider the region where $\mathcal{N}(\fb)\leq B$.  Here our assumptions on the discriminant give $\abs{\Delta_F}\leq X/\mathcal{N}(\fb)$.  Hence we are able to apply \Cref{lem:LDTT} to the inner count in this region, obtaining
\begin{equation}\label{eq:firstmu2expression}
\begin{split}
&\sum_{\substack{\fb \leq \cO_F \\ \cN(\fb) \leq B}}\gamma_{F,M}(\fb)
\#\set{\fa\leq \cO_F~:~\cN(\fa)\leq \frac{X}{\cN(\fb)}}
\\=&
\zeta_F^*(1)X\sum_{\substack{\fb \leq \cO_F \\ \cN(\fb) \leq B}}\frac{\gamma_{F,M}(\fb)}{\cN(\fb)}
+O_d\braces{\sum_{\substack{\fb\leq \cO_F\\\cN(\fb)\leq B}}\abs{\gamma_{F,M}(\fb)}(X/\cN(\fb))^{\frac{d}{d+1}}\log(X)^{d-1}}.
\end{split}
\end{equation}
For the error term, we apply partial summation and \Cref{lem:gammabound} to obtain that
\[\sum_{\substack{\fb\leq \cO_F\\\cN(\fb)\leq B}}\frac{\abs{\gamma_{F,M}(\fb)}}{\cN(\fb)^{\frac{d}{d+1}}}\ll_d 2^{d\omega(M)}B^{\frac{1}{2}-\frac{d}{d+1}}\log(B)^{2^{d}d^2-1}.\]
Hence \eqref{eq:firstmu2expression} is equal to
\begin{equation}\label{eq:secondmu2expression}
\zeta_F^*(1)X\sum_{\substack{\fb \leq \cO_F \\ \cN(\fb) \leq B}}\frac{\gamma_{F,M}(\fb)}{\cN(\fb)}
+O_d\braces{2^{d\omega(M)}X^{\frac{1}{2}}\log(X)^{2^{d}d^2+d-2}}.
\end{equation}
\Cref{lem:gammabound} and an elementary tail estimate then show that this gives the proposed main term for \Cref{eq:prop_counting_ideals}, and acceptable error.

It then remains to show that the region $\mathcal{N}(\fb)>B$ in \eqref{eq:prop_counting_ideals} contributes only to our proposed error.  Note that the number of ideals $\fa$ with fixed norm $a$ is at most $\tau(a)^d$, see for example \cite{MR195803}*{p220}, so that in our region we have
\begin{align*}
    \#\set{\fa\leq \cO_F~:~\cN(\fa)\leq \frac{X}{\cN(\fb)}}
    &\leq \#\set{\fa\leq \cO_F~:~\cN(\fa)\leq X^\epsilon}
    \\&=\sum_{a\leq X^\epsilon}\#\set{\fa\leq \cO_F~:~\cN(\fa)=a}
    \\&\leq \sum_{a\leq X^\epsilon}\tau(a)^d
    \\&\ll X^\epsilon \log(X)^{2^d-1},
\end{align*}
where the final bound is \cite{MR2061214}*{p23 (1.80)}.  Hence, via \Cref{lem:gammabound} the contribution in the region $\mathcal{N}(\fb)>B$ is 
\begin{align*}
    \sum_{\substack{\fb \leq \cO_F \\ X^{1-\epsilon}< \cN(\fb) \leq X}}&\gamma_{F,M}(\fb)
    \#\set{\fa\leq \cO_F~:~\cN(\fa)\leq \frac{X}{\cN(\fb)}}
    \\&\ll X^{\epsilon}\log(X)^{2^d-1}\sum_{\substack{\fb \leq \cO_F \\ X^{1-\epsilon}< \cN(\fb) \leq X}}\abs{\gamma_{F,M}(\fb)}
    \\&\ll_d 2^{d\omega(M)}X^{1/2+\epsilon}\log(X)^{2^d(d^2+1)-1}.
\end{align*}
Thus the result holds.
\end{proof}

It is often useful, and indeed is in this article, to be able to also stipulate that the ideal represents a class in some specific subgroup of the ideal class group.  We now present this refined count, which is the main result of this section.

\begin{theorem}\label{thm:idealcounting}
    Let $d\in\ZZ_{>0}$ and $\epsilon>0$.  For every number field $F$ of degree $d$, integer $M\geq 1$, subgroup $H\leq \Cl_F$, and real number $X\geq\abs{\Delta_F}^{1/\epsilon}$, we have
    \begin{align*}
    & \#\set{\fa\leq \cO_F~:~
    \begin{array}{l}
        \cN(\fa)\textnormal{ squarefree}, \gcd(\cN(\fa),M)=1, \\
            \,[\fa]\in H \leq \Cl_F, \cN(\fa)\leq X
    \end{array}} \\
    &=  \frac{\#H}{\#\Cl_F}\zeta^*_F(1)L(\gamma_{F,M},1)X+
    O_{d,\epsilon}\braces{2^{d\omega(M)}X^{1/2+\epsilon} \log(X)^{2^{d}(d^2+1)-1}},
    \end{align*}
where the implied constant depends on the degree $d$,  but is independent of $F$, $M$, $H$, and $X$.
\end{theorem}
\begin{proof}
We apply character orthogonality to obtain
\begin{equation*} \label{eqn:character_class_group}
\sum_{\substack{\fa\leq \cO_F \\ [\fa] \in H \\
\cN(\fa)\leq X}} \mu^2(\cN(\fa))\one_{M}(\cN(\fa))=\frac{\#H}{\#\Cl_F}\sum_{\chi\in(\Cl_F/H)^\vee}\sum_{\substack{\fa\leq \cO_F \\
\cN(\fa)\leq X}} \mu^2(\cN(\fa))\one_{M}(\cN(\fa))\chi(\fa).
\end{equation*}
For the trivial character, by Proposition \ref{prop:sqfideals} one obtains the stated main term and satisfactory error term. Thus it suffices to show that the nontrivial characters contribute to the error term. We use the identity $\mu^2(\cN(\fa))\one_{M}(\cN(\fa))=\sum_{\fb\mid\fa}\gamma_{F,M}(\fb)$ (see \Cref{lem:inverseconv}) to obtain that the inner sum  equals
\begin{align*}
\sum_{\substack{\fa\leq \cO_F \\
\cN(\fa)\leq X}} \mu^2(\cN(\fa))\one_{M}(\cN(\fa))\chi(\fa)
= & \sum_{\substack{\fb \leq \cO_F \\ \cN(\fb) \leq X}}\gamma_{F,M}(\fb)\sum_{\substack{\fa\leq \cO_F \\ \fb \mid \fa \\ 
\cN(\fa)\leq X}} \chi(\fa)\\
= & \sum_{\substack{\fb \leq \cO_F \\ \cN(\fb) \leq X}}\gamma_{F,M}(\fb)\chi(\fb)
\sum_{\substack{\fa\leq \cO_F \\ 
\cN(\fa)\leq X/\cN(\fb)}} \chi(\fa).
\end{align*}
Apply Landau--Polya--Vinogradov \cite{LPV}*{p479} to obtain
\[\sum_{\substack{\fa\leq \cO_F \\ \cN(\fa)\leq X/\cN(\fb)}} \chi(\fa)\ll_d \abs{\Delta_F}^{\frac{1}{d+1}}\log(\abs{\Delta_F})^{d}\braces{\frac{X}{\cN(\fb)}}^{\frac{d-1}{d+1}}\ll X^{\frac{d-1+\epsilon}{d+1}}\log(X)^d\frac{1}{\mathcal{N}(\fb)^{\frac{d-1}{d+1}}}.\]
Using \Cref{lem:gammabound} and partial summation, uniformly in $F$ and $M$ we have
\[\sum_{\substack{\fb\leq\cO_F\\\cN(\fb)\leq X}}\frac{\abs{\gamma_{F,M}(\fb)}}{\cN(\fb)^{\frac{d-1}{d+1}}}\ll _d 2^{d\omega(M)} 
X^{\frac{1}{2}-\frac{d-1}{d+1}}\log(X)^{2^{d}d^2-1}.\]
Substituting back in for the nontrivial characters gives the stated error term.
\end{proof}

\section{Average residue for cyclic cubic zeta functions} \label{sec:cubic}
We now prepare for the proof of Theorem \ref{thm:LTAvg}.  In fact, we will prove the following stronger version.
\begin{theorem}\label{thm:LTAvg_d}
    Let $D> 0$. Let $d \in \N$ be squarefree which is only divisible by primes that
    are $1 \bmod 3$ and with $0< d \leq (\log X)^D$. Then we have
    \[\sum_{\substack{F\textnormal{ cyclic cubic}\\\textnormal{unram. at }3\\ d \mid \Delta_F \\ \Delta_F\leq X}}\zeta_F^*(1)=\frac{7\pi}{26\sqrt{3}}\alpha\beta_d \cdot\frac{2^{\omega(d)}}{d}\cdot X^{1/2} + O_{D}(X^{1/2}\log(X)^{-D}),\]
    where 
    \begin{align*}
    \alpha&=\prod_{p\equiv 1\bmod 3} \left(1-\frac{1}{p}\right)^2
    \left(1 +\frac{2}{p} + \frac{1 +3p^{-1} + p^{-2} - p^{-6}}{p^2(1 - p^{-3})^2}\right)
    \prod_{p\equiv 2\bmod 3}\left({\frac{1+p^{-3}}{1-p^{-3}}}\right) \\
    \beta_d&=
    \prod_{p \mid d}\left(\frac{1 - p^{-1} + p^{-2}}{1 + p^{-1} + p^{-2}}\right)
    \left(1-\frac{1}{p}\right)^{-2}
    \left(1 +\frac{2}{p} + \frac{1 +3p^{-1} + p^{-2} - p^{-6}}{p^2(1 - p^{-3})^2}\right)^{-1}.
    \end{align*}
\end{theorem}
Since \Cref{thm:LTAvg} is the $d=1$ case of \Cref{thm:LTAvg_d}, it is sufficient to prove this more flexible result.

The proof will proceed by expanding the residue $\zeta_F(1)^*$ as a product of Dirichlet series, truncating these and expanding the product to obtain a character sum in $3$ variables.  We then use a combination of bilinear sieve and Siegel--Walfisz type results to exploit the oscillation of the characters and isolate the main term, and then evaluate this by standard Dirichlet series techniques.  We will begin by gathering the necessary preliminaries for our approach, before going on to prove \Cref{thm:LTAvg_d}.

\subsection{Preliminaries}
\subsubsection{Cubic characters}
We first re-express the Dirichlet characters associated to cubic fields as cubic residue symbols.
\begin{lemma}\label{lem:characterchange}
    Let $q$ be a squarefree integer coprime to $3$.
    Then the primitive cubic Dirichlet characters of conductor $q$ 
    are exactly the characters of the form $\art{\cdot}{I}_3$ for an ideal
    $I \subset \Z[\zeta_3]$  of norm $q$.  Moreover, the discriminant of the associated abelian extension
    is $\cN(I)^2 = q^2$.% \dan{Presumably $\cN(I)^2 = q^2$ by definition?}
\end{lemma}
\begin{proof}
    By multiplicativity of characters, we can reduce to the case where $q$ is prime, hence $q \equiv 1 \bmod 3$. 
    Such Dirichlet characters exactly correspond to homomorphisms
    $\FF_q^\times \to \CC^\times$ of order $3$. One then takes $I$
    to be an ideal of $\Z[\zeta_3]$ above $q$, and observes that since $q\equiv 1\bmod 3$ it must split in $\QQ(\zeta_3)$ and hence
    $\Z[\zeta_3]/I \cong \FF_q$. 
    The last part follows from the conductor-discriminant formula.
\end{proof}

We will approach the averaging problem by working with Dirichlet $L$-series.

\begin{lemma}\label{lem:Ltruncated}
    Let $\chi$ be a primitive Dirichlet character of conductor $q$. Then for $X\geq 1$
    \[L(\chi,1)=\sum_{n\leq X}\frac{\chi(n)}{n}+O\braces{\frac{\sqrt{q}\log(q)}{X}}.\]
\end{lemma}
\begin{proof}
    Apply partial summation and Polya-Vinoradov to the tail.
\end{proof}

\subsubsection{Bilinear sieve}
We shall deal with large $n$ using the following bilinear sieve result.
\begin{proposition}\label{prop:bilinearsieve}  
    Let $a_I,b_n$ be complex numbers of absolute value at most $1$ as $I$ ranges over ideals of $\ZZ[\zeta_3]$ and $n$ ranges over positive integers.  There exists $\delta>0$ such that for all $A>0$ and $X\gg_A1$ we have 
    \[\sum_{\substack{\log(X)^A\leq n\leq X}}
        \sum_{\substack{I\leq \ZZ[\zeta_3]\\\cN(I)\textnormal{ sqfree}\\\gcd(\mathcal{N}(I),3)=1\\\cN(I)\leq X}}
        a_Ib_n\frac{1}{n}
        \art{n}{I}_3
      \ll X\log(X)^{\delta^{-1}-\delta A}.\]
\end{proposition}
\begin{proof}
    We claim that the function
    \begin{align*}
    \eta:I_{\QQ(\zeta_3)}\times I_\QQ&\to \CC^\times\cup\set{0}\\
    \eta(I,n)&=\begin{cases}
        0&\text{if }3\mid \cN(I),\\
        \art{n}{I}_3&\text{else,}\\
    \end{cases}\
    \end{align*}
    is an $(A,q)$-oscillating bilinear character for some constants $A,q$ in the sense of \cite{MR4564991}*{Def.~4.19}, so that we can appeal to the bilinear sieve estimate there.  Indeed, as in \Cref{lem:characterchange}, for $I$ such that $\cN(I)$ is squarefree and coprime to $3$, the function $\art{\cdot}{I}_3$, is a Dirichlet character of conductor $\cN(I)$.  Moreover, via class field theory, for cubefree $n$ the function $\art{n}{\cdot}_3$ is a non-principal Hecke character on $\gal(\QQ(\sqrt[3]{n},\zeta_3)/\QQ(\zeta_3))$ of conductor at most a fixed (independent of $n$) power of $3n$

    Now appealing to the bilinear sieve result in \cite{MR4564991}*{Prop.~4.21}, there is a constant $\delta>0$ such that for all $X,T\geq 2$:
    \[Z(X,T):=\sum_{\substack{I\leq \ZZ[\zeta_3]\\\cN(I)\leq X\\\gcd(\cN(I),3)=1\\\cN(I)\textnormal{ sq. free}}}
    \sum_{\substack{1\leq n\leq T}}
    a_Ib_n
    \art{n}{I}_3\ll 
    TX\log(XT)^{\delta^{-1}}\max\set{T^{-\delta},X^{-\delta}}.\]
    Hence we apply partial summation to the sum over $n$ in the lemma statement to obtain
    \begin{align*}
    \sum_{\substack{\log(X)^A\leq n\leq X}}
        &\sum_{\substack{I\leq \ZZ[\zeta_3]\\\gcd(\cN(I),3)=1\\\cN(I)\leq X}}
        a_Ib_n\frac{1}{n}
        \art{n}{I}_3
        \\&=\frac{1}{X}Z(X,X)-\frac{1}{\log(X)^A}Z(X,\log(X)^A)-\int_{\log(X)^A}^{X}\frac{Z(X,t)}{t^2}dt
        \\&\ll
        X^{1-\delta}\log(X)^{\delta^{-1}} +
        X\log(X)^{\delta^{-1}-\delta A} +
        X\log(X)^{\delta^{-1}}\int_{\log(X)^A}^X \frac{1}{t^{1+\delta}}dt
        \\&\ll
        X\log(X)^{\delta^{-1}-\delta A}. \qedhere
    \end{align*}
\end{proof}

\subsubsection{Siegel-Walfisz/LSD}
To deal with small $n$, we will make use of a Siegel--Walfisz-type result proved using the LSD method.

\begin{lemma}[\cite{KPSS}*{Lemma 4.8}]\label{lem:KPSS_LSD}
Let $C>0$.  For every $X>1$, every $1 \leq \Delta \leq X^C$, and all non-principal Hecke characters $\chi: I_{\Z[\zeta_3]} \to \CC$ with modulus $\fm$ and conductor $f_{\chi}\mid \fm$ such that $\mathcal{N}(f_{\chi}) \omega(\mathcal{N}(\mathfrak{m})) \leq (\log X)^C$ we have the bound
\[
\sum_{\substack{I \in I_{\Z[\zeta_3]} \\ \mathcal{N}(I) \leq X}} \mu^2(\mathcal{N}(I)\cdot \Delta) \chi(I) \ll_C X (\log X)^{-C}
\]
where the implied constant depends only on $C$.
\end{lemma}
We apply this as follows.
\begin{lemma}\label{lem:SW}
    Let $C>0$.  Then for every integer $n$ such that $n\leq \log(X)^C$ and $n$ is not a perfect cube, and every integer $1\leq d\leq \log(X)^C$ which is squarefree and a product of primes which are $1\bmod 3$, we have
    \[\sum_{\substack{I\leq \ZZ[\zeta_3]\\\gcd(\cN(I),3d)=d\\\cN(I)\leq X}}\mu^2(\cN(I))\art{n}{I}_3\ll_C X\log(X)^{-C},\]
    where the implied constant is uniform in $d,n,$ and $X$.
\end{lemma}
\begin{proof}
    Firstly, note that
    \begin{align*}
    \sum_{\substack{I\leq \ZZ[\zeta_3]\\\gcd(\cN(I),3d)=d\\\cN(I)\leq X}}\mu^2(\cN(I))\art{n}{I}_3
    &=\sum_{\substack{J\leq \ZZ[\zeta_3]\\\cN(J)=d\\\cN(J)\leq X}} \art{n}{J}_3
      \sum_{\substack{I\leq \ZZ[\zeta_3]\\\gcd(\cN(I),3d)=1\\\cN(I)\leq X/d}}\mu^2(\cN(I))\art{n}{I}_3
    \\&\ll2^{\omega(d)}
      \sum_{\substack{I\leq \ZZ[\zeta_3]\\\gcd(\cN(I),3d)=1\\\cN(I)\leq X/d}}\mu^2(\cN(I))\art{n}{I}_3
    \\&\ll_{\epsilon,C}\log(X)^{\epsilon}
      \sum_{\substack{I\leq \ZZ[\zeta_3]\\\gcd(\cN(I),3d)=1\\\cN(I)\leq X/d}}\mu^2(\cN(I))\art{n}{I}_3
    \end{align*}
    where we've used that $\#\set{J\leq \ZZ[\zeta_3]~:~\mathcal{N}(J)=d}=2^{\omega(d)}\ll_{\epsilon,C} d^{\epsilon/C}\leq \log(X)^\epsilon$.
    
    Note that $\chi(I):=\art{n}{I}_3$ is a Hecke character for the Galois group $\gal(\QQ(\zeta_3,\sqrt[3]{n})/\QQ(\zeta_3))$.  Since $n$ is not a perfect cube, this is a non-principal character.  Moreover, its modulus $\fm$ is supported on the primes dividing $3n$, and its conductor $f_\chi$ also is at most a constant power of $3n$.  In particular, 
    \[\cN(f_\chi)\omega(\cN(\fm))\ll (n\omega(n))^{O(1)}\leq n^{O(1)}\leq \log(X)^{O(C)}.\]
    After potentially enlarging $C$, we may then assume that $\cN(f_\chi)\omega(\cN(\fm))\leq \log(X)^{C}$.

    We now apply \Cref{lem:KPSS_LSD} to the inner sum above, to obtain
    \[\sum_{\substack{I\leq \ZZ[\zeta_3]\\\gcd(\cN(I),3d)=1\\\cN(I)\leq X/d}}\mu^2(\cN(I))\art{n}{I}_3
    =\sum_{\substack{I\leq \ZZ[\zeta_3]\\\cN(I)\leq X/d}}\mu^2(\cN(I)\cdot 3d)\art{n}{I}_3
    \ll_C X\log(X)^{-C}.\]
    Possibly enlarging $C$ again to account for $\epsilon$ above, we now have the result.
\end{proof}

\subsubsection{Counting cyclic cubic fields}
We require the asymptotic
\begin{equation} \label{eqn:number_F}
    \#\{F/\Q \textnormal{ cyclic cubic}: F \textnormal{ unramified at }3, \abs{\Delta_F}\leq X\}
    \sim C X^{1/2},
\end{equation}
for some constant $C>0$. This follows for example from  Wright \cite{MR969545}*{Thm.~I.4}.

\subsection{Proof of \Cref{thm:LTAvg_d}}
Our toolkit has been filled, and so we are ready to prove the main result of this section. 

\subsubsection{Reduction to character sum}
To begin, we reformulate into a character sum problem via \Cref{lem:Ltruncated}.  Writing $\chi_F$ for a choice of Dirichlet character associated to $F$, this yields

\begin{align*}
    \sum_{\substack{F\textnormal{ cyclic cubic}\\\textnormal{unram. at }3\\ d \mid \Delta_F \\ \abs{\Delta_F}\leq X}}\zeta_F^*(1)
    &=\sum_{\substack{F\textnormal{ cyclic cubic}\\\textnormal{unram. at }3\\ d \mid \Delta_F \\\abs{\Delta_F}\leq X}}L(\chi_F,1)L(\bar{\chi}_F,1)
    \\&=\sum_{\substack{F\textnormal{ cyclic cubic}\\\textnormal{unram. at }3\\ d \mid \Delta_F \\\abs{\Delta_F}\leq X}}\sum_{\substack{1\leq n_1\leq X^{1/2} \\ 1\leq n_2\leq X^{1/2}}}\frac{\chi_F(n_1)\bar{\chi_F(n_2)}}{n_1n_2} + O\braces{X^{1/4}\log(X)^2}
    \\&=\frac{1}{2}\sum_{\substack{I\leq \ZZ[\zeta_3]\\\gcd(\cN(I),3d)=d\\\cN(I)\leq X^{1/2}}}
    \sum_{\substack{1\leq n_1\leq X^{1/2} \\ 1\leq n_2\leq X^{1/2}}}\mu^2(\cN(I))\frac{\art{n_1}{I}_3\bar{\art{n_2}{I}_3}}{n_1n_2}+O\braces{X^{1/4}\log(X)^2}
\end{align*}
where in the second equality we used \eqref{eqn:number_F}, and in the third we apply \Cref{lem:characterchange}.  
Defining
\[S_d(X):=\sum_{\substack{I\leq \ZZ[\zeta_3]\\\gcd(\cN(I),3d)=d\\\cN(I)\leq X}}
\sum_{\substack{1\leq n_1\leq X^{1/2} \\ 1\leq n_2\leq X^{1/2}}}\mu^2(\cN(I))\frac{\art{n_1}{I}_3\bar{\art{n_2}{I}_3}}{n_1n_2},\]
we have then shown that
\begin{equation}\label{eq:charsum}
\sum_{\substack{F\textnormal{ cyclic cubic}\\\textnormal{unram. at }3\\ d \mid \Delta_F \\ \abs{\Delta_F}\leq X}}\zeta_F^*(1)=\frac{1}{2}S_d(X^{1/2}) + O\braces{X^{1/4}\log(X)^2}.
\end{equation}
In particular, it is now sufficient to study the character sum $S_d(X)$ in order to prove \Cref{thm:LTAvg_d}.
\subsubsection{Isolating the main term in $S_d$}
We now study the sum $S_d(X)$, recalling that we have $d\leq \log(X)^D$ for some $D>0$.  Let us separate this sum depending on the sizes of the $n_i$.

\begin{definition}\label{def:SM}
    For $\cM\subseteq \set{1,2}$ we define $S^\cM_d(X)$ to be the subsum of $S_d(X)$ such that
    \begin{itemize}
        \item for $m\in \cM$, $n_m>\log(X)^A$; and
        \item for $m\not\in\cM$, $n_m\leq \log(X)^A$. 
    \end{itemize}
    Here, $A>0$ is a (large) exponent to be chosen later.  Clearly, $S_d(X)=\sum_{\cM\subseteq \set{1,2}}S_d^{\cM}(X)$.
\end{definition}

\begin{lemma}\label{lem:deltaappears}
    If $\cM\neq\emptyset$ then there exists a constant $\delta>0$ such that $S_d^{\cM}(X)\ll X\log(X)^{1+\delta^{-1}-\delta A}$.
\end{lemma}
\begin{proof}
    By symmetry, let us assume that $1\in \cM$.  Now, for each $n_2$, we can apply \Cref{prop:bilinearsieve} to bound
    \[\sum_{\substack{I\leq \ZZ[\zeta_3]\\\gcd(\cN(I),3d)=d\\\cN(I)\leq X}}\sum_{\log(X)^A< n_1\leq X}\mu^2(\cN(I))\frac{\art{n_1}{I}_3\bar{\art{n_2}{I}_3}}{n_1}\ll X\log(X)^{\delta^{-1}-\delta A}.\]
    In particular
    \[S_d^{\cM}(X)\ll X\log(X)^{\delta^{-1}-\delta A}\sum_{1\leq n_2\leq X}\frac{1}{n_2}\ll X\log(X)^{1+\delta^{-1}-\delta A}. \qedhere\]
\end{proof}

So if we take large enough $A$ then we will reduce to studying $S_d^\emptyset(X)$.  Now, when the character in $S_d^{\emptyset}(X)$ is nontrivial, we will use our Siegel--Walfisz result (\Cref{lem:SW}) to recover a saving.

\begin{lemma}\label{lem:Semptyestimate}  We have
\[S_d^\emptyset(X)=\frac{\pi}{3\sqrt{3}}L(\gamma_{\QQ(\zeta_3),3},1) \cdot  \frac{2^{\omega(d)} X}{d}
\sum_{\substack{1\leq n_1\leq \log(X)^A\\1\leq n_2\leq \log(X)^A\\n_1n_2^2\in\QQ^{\times3}}}
    \frac{\eta(n_1n_2d)}{n_1n_2}
    +O\braces{X\log(X)^{-A}},\]
    where $\eta$ is the multiplicative function on $\N$ given on non-trivial prime powers $p^k$ by
\[\eta(p^k)=\begin{cases}(1 + 2p^{-1})^{-1},&\textnormal{if }p\equiv 1\bmod 3,\\1,&\textnormal{else}.\end{cases}\]
\end{lemma}
\begin{proof}
    By definition
    \begin{align*}
    S_d^\emptyset(X)
    &:=\sum_{\substack{1\leq n_1\leq \log(X)^A\\1\leq n_2\leq \log(X)^A}}
    \frac{1}{n_1n_2}\sum_{\substack{I\leq \ZZ[\zeta_3]\\\gcd(\cN(I),3d)=d\\\cN(I)\leq X}}\mu^2(\cN(I))
        \art{n_1}{I}_3\bar{\art{n_2}{I}_3}.
    \end{align*}
    We next use  $\bar{\art{\cdot}{I}_3}=\art{\cdot}{I}_3^2$ and apply \Cref{lem:SW} to 
    the collection of $n_1,n_2$ such that $n_1n_2^2$ is not a perfect cube, to obtain   
    \begin{align*}
    S_d^\emptyset(X)
    &=\sum_{\substack{1\leq n_1\leq \log(X)^A\\1\leq n_2\leq \log(X)^A\\n_1n_2^2\in\QQ^{\times 3}}}
    \frac{1}{n_1n_2}\sum_{\substack{I\leq \ZZ[\zeta_3]\\\gcd(\cN(I),3d)=d\\\cN(I)\leq X}}\mu^2(\cN(I))\art{n_1n_2^2}{I}_3 + O\braces{X\log(X)^{-A}}
    \\&=\sum_{\substack{1\leq n_1\leq \log(X)^A\\1\leq n_2\leq \log(X)^A\\n_1n_2^2\in\QQ^{\times 3}}}
    \frac{1}{n_1n_2}\sum_{\substack{I\leq \ZZ[\zeta_3]\\\gcd(\cN(I),3n_1n_2d)=d\\\cN(I)\leq X}}\mu^2(\cN(I)) + O\braces{X\log(X)^{-A}}.
    \end{align*}
    To encode the condition $d \mid \cN(I)$, we sum over all squarefree ideals $J$ with $\cN(J) = d$.
    Then we have $J \mid I$ for some such $J$. After a change of variables, this yields
    \begin{align*}
    S_d^\emptyset(X) &=\sum_{\substack{1\leq n_1\leq \log(X)^A\\1\leq n_2\leq \log(X)^A\\n_1n_2^2\in\QQ^{\times 3}}}  \frac{1}{n_1n_2}
    \sum_{\substack{J\leq \ZZ[\zeta_3]\\ \cN(J) = d}}
    \sum_{\substack{I\leq \ZZ[\zeta_3]\\\gcd(\cN(I),3n_1n_2d)=1 \\\cN(I)\leq X/d}}\mu^2(\cN(I)) + O\braces{X\log(X)^{-A}}.
    \end{align*}    
    We now apply \Cref{prop:sqfideals} to the inner sum, to obtain
    \begin{align*}
    S_d^\emptyset(X)
    &=\frac{2^{\omega(d)} X}{d}
    \sum_{\substack{1\leq n_1\leq \log(X)^A\\1\leq n_2\leq \log(X)^A\\n_1n_2^2\in\QQ^{\times3}}}
    \frac{\zeta_{\QQ(\zeta_3)}^*(1)L(\gamma_{\QQ(\zeta_3),3n_1n_2d},1)}{n_1n_2}
    \\&+O\braces{X\log(X)^{-A} + 2^{\omega(d)}X^{1/2}\log(X)^{19}\sum_{\substack{1\leq n_1\leq \log(X)^A\\1\leq n_2\leq \log(X)^A\\n_1n_2^2\in\QQ^{\times3}}}4^{\omega(n_1n_2d)}}
    \end{align*}
    where we use that $\#\{J\leq \ZZ[\zeta_3] :  \cN(J) = d\} = 2^{\omega(d)}$.
    An elementary divisor bound, together with $d\leq \log(X)^D$, gives 
    \[\sum_{\substack{1\leq n_1\leq \log(X)^A\\1\leq n_2\leq \log(X)^A\\n_1n_2^2\in\QQ^{\times3}}}4^{\omega(n_1n_2d)}
    \ll_\epsilon  \sum_{\substack{1\leq n_1\leq \log(X)^A\\1\leq n_2\leq \log(X)^A\\n_1n_2^2\in\QQ^{\times3}}}n_1^\epsilon n_2^\epsilon d^{\epsilon}\ll_\epsilon\log(X)^{2A(1+\epsilon) + \epsilon D},\]
    so in fact choosing a fixed $\epsilon$ and noting that $\zeta_{\QQ(\zeta_3)}^*(1) = \pi/3\sqrt{3}$ we have
    \[S_d^\emptyset(X)
    =\frac{\pi}{3\sqrt{3}} \cdot  \frac{2^{\omega(d)} X}{d}
    \sum_{\substack{1\leq n_1\leq \log(X)^A\\1\leq n_2\leq \log(X)^A\\n_1n_2^2\in\QQ^{\times3}}}
    \frac{L(\gamma_{\QQ(\zeta_3),3n_1n_2d},1)}{n_1n_2}+O\braces{X\log(X)^{-A}}.\]
    We then apply Lemma \ref{lem:L_gamma_3} to find that $L(\gamma_{\QQ(\zeta_3),3n_1n_2d},1)$
    equals
    \begin{align*}
    &\frac{2}{3}
    \prod_{\substack{p\equiv 1 \bmod 3 \\ p \nmid n_1n_2d}}\left( 1 - \frac{1}{p}\right)^2\left(1 + \frac{2}{p}\right)
    \prod_{\substack{p\equiv 1 \bmod 3\\p\mid n_1n_2 d}}\left( 1 - \frac{1}{p}\right)^2 
    \prod_{p\equiv 2 \bmod 3}\braces{1-\frac{1}{p^2}}
    \\&=\eta(n_1n_2d)L(\gamma_{\QQ(\zeta_3),3},1)\end{align*}
    and so the result follows.
\end{proof}

\subsubsection{Main term analysis}
Let us now conclude our discussion on character sums by applying the lemmata above to give a precise result for the sum $S(X)$ from \eqref{eq:charsum}.

\begin{proposition}\label{prop:SX}
    For every $C>0$ we have
    \[S_d(X) =\frac{7\pi}{26\sqrt{3}}\ \alpha \beta_d \cdot \frac{2^{\omega(d)}X}{d} 
    +O_C\braces{X\log(X)^{-C}}\]
    where $\alpha$ and $\beta_d$ are as in Theorem \ref{thm:LTAvg_d}.
\end{proposition}
\begin{proof}
    Choosing $A$ in \Cref{def:SM} large enough so that $1+\delta^{-1}-A\delta<-C$, where $\delta$ is as in \Cref{lem:deltaappears}, it follows from said lemma that
    \begin{align*}
    S_d(X)&=S_d^{\emptyset}(X)+O\braces{X\log(X)^{-C}} \\
    &=\frac{\pi}{3\sqrt{3}}L(\gamma_{\QQ(\zeta_3),3},1)
    \cdot \frac{2^{\omega(d)}X}{d} 
    \sum_{\substack{1\leq n_1\leq \log(X)^A\\1\leq n_2\leq \log(X)^A\\n_1n_2^2\in\QQ^{\times3}}}
    \frac{\eta(n_1n_2 d)}{n_1n_2}
    +O\braces{X\log(X)^{-C}}
    \end{align*}
    where the second equality follows from \Cref{lem:Semptyestimate}, after at worst enlarging $A$ again so that $A>C$.  It then remains to understand this sum of $\eta$.  We can write every integer uniquely as $n_1=d_1d_2^2m_1^3$ for integers $d_1,d_2,m_1$ with $\mu^2(d_1d_2) = 1$. The constraint that $n_1n_2^2$ is a cube means that we have $n_2=d_1d_2^2m_2^3$ for some integer $m_2$. 
    Hence this change of variables allows us to write
    \begin{equation}\label{eq:etaeq}
    \sum_{\substack{1\leq n_1\leq \log(X)^A\\1\leq n_2\leq \log(X)^A\\n_1n_2\in\QQ^{\times3}}}
    \frac{\eta(n_1n_2d)}{n_1n_2} 
    = \sum_{\substack{1\leq m_1\leq \log(X)^{A/3}\\1\leq m_2\leq \log(X)^{A/3}}}
    \sum_{\substack{1\leq d_1\leq \log(X)^{A}\\1\leq d_2\leq \log(X)^{A/2}}}
    \frac{\mu^2(d_1d_2)\eta(d_1^2d_2^4m_1^3m_2^3 d)}{d_1^2d_2^4m_1^3m_2^3}.
    \end{equation}
    Note that this sum is absolutely convergent in $X$ since $\abs{\eta(N)}\leq1$.
    As $X\to\infty$ we can expand this out into an Euler product arising from a 
    multiple Dirichlet series which converges to
    \begin{align*}
    &\prod_{p}\braces{\sum_{d_1,d_2,m_1,m_2=1}^\infty \frac{\mu^2(p^{d_1+d_2})
    \eta(p^{2d_1 + 4d_2 + 3m_1 + 3m_2})}{p^{2d_1 + 4d_2 + 3m_1 + 3m_2 + v_p(d)}}}. 
    \end{align*}
    We first consider those $p \nmid$. 
    On using the explicit formula for $\eta$ from Lemma \ref{lem:Semptyestimate}, this yields    
    \begin{align*}
    &\prod_{p \nmid d}\braces{1+\eta(p)\braces{\braces{1+p^{-2} + p^{-4}}\braces{1+p^{-3}+p^{-6}+\dots}^2-1}}
    \\&=\prod_{p \nmid d}\braces{1+\eta(p)\braces{\braces{1+p^{-2} +p^{-4}}\braces{1-p^{-3}}^{-2}-1}}
    \\&=\prod_{\substack{p\not\equiv 1\bmod 3 \\ p \nmid d}}\braces{1+p^{-2} +p^{-4}}\braces{1-p^{-3}}^{-2}
    \prod_{\substack{p\equiv 1\bmod 3 \\ p \nmid d}} \left(1 + \frac{1 +3p^{-1} + p^{-2} - p^{-6}}
    {p^2(1 + 2p^{-1})(1 - p^{-3})^2}\right).
    \end{align*}
    Similarly, for $p \mid d$ we obtain
    \begin{align*}
    &\prod_{\substack{p \mid d}}(1 + 2p^{-1})^{-1}\braces{1+p^{-2} +p^{-4}}\braces{1-p^{-3}}^{-2}
    \end{align*}    
    since $d$ is only divisible by primes which are $1 \bmod 3$.
    Using \Cref{lem:L_gamma_3}  we obtain
    \[L(\gamma_{\QQ(\zeta_3),3},1) = \frac{2}{3}\prod_{\substack{p\equiv 1 \bmod 3}}\left( 1 - \frac{1}{p}\right)^2\left( 1 + \frac{2}{p}\right) \prod_{p\equiv 2 \bmod 3}\braces{1-\frac{1}{p^2}}.\]
    Recalling the factor of $\frac{\pi}{3\sqrt{3}}$ and making explicit the Euler factor at $p =3$,
    a short argument cancelling terms in Euler factors shows that we obtain the claimed leading term.

    Finally, the tail (i.e.~the difference between the limit for the multiple Dirichlet series and the sum in \Cref{eq:etaeq}) is bounded by an integral estimate which is at worst $\log(X)^{-\frac{23}{6}A}\ll \log(X)^{-C}$, hence the claim holds.
\end{proof}
\subsubsection{Proof of \Cref{thm:LTAvg_d}}
This now follows from \Cref{eq:charsum} and \Cref{prop:SX}. \qed

\section{Counting \texorpdfstring{$A_4$}{A4}-quartics} \label{sec:quartic}
We now turn to the central purpose of this article:  counting $A_4$-quartic extensions of $\QQ$. We begin by preparing for the proof of Theorem \ref{thm:fixed_cubic} on $A_4$-quartics with given cubic resolvent. 

\subsection{Fields to Ideals}
We will require some results from \cite{MR3215550} about parametrising $A_4$-quartics with given cubic resolvent.

\begin{proposition}[\cite{MR3215550}*{Thm.~2.2}]\label{prop:CTThm2.2}
    Let $F$ be a cyclic cubic field.  There is a $3:1$ surjective map
    \[\set{K/F~:~\substack{\exists\alpha\in F^\times\backslash{F^{\times2}}\textnormal{ such that}\\K=F(\sqrt{\alpha})\textnormal{ and }N_{F/\QQ}(\alpha)\in\QQ^{\times 2}}}\to \set{L/\QQ~:~\substack{A_4\textnormal{-quartic with}\\\textnormal{cubic resolvent }F}}.\]
    Under this correspondence, we have $\Delta_L=\cN(\fd(K/F))\Delta_F$, where $\fd$ denotes the relative discriminant.  Under this map, the elements of the fibre above an $A_4$-quartic $L$ have the same Galois closure as $L$.
\end{proposition}

We recall that the $2$-Selmer group of a number field $k$ is defined to be
\begin{equation*}
    \sel{2}(k) := \{ x \in k^\times/k^{\times 2} : v(x) \equiv 0 \bmod 2 \text{ for all non-archimedean places } v\}.
\end{equation*}

\begin{proposition}[\cite{MR3215550}*{Prop.~3.7}]\label{prop:CTProp3.7}
    Let $F$ be a number field.  There is a bijection between
    \[
    \set{(\fa,u)~:~\substack{\fa\leq \cO_F\textnormal{ is squarefree with square norm}\\ [\fa]\in 2\Cl_F\\u\in\ker\braces{N_{F/\QQ}:\sel{2}(F)\to \sel{2}(\QQ)}}}
    \leftrightarrow
    \set{K/F~:~\substack{\exists\alpha\in F^\times\backslash{F^{\times2}}\textnormal{ such that}\\K=F(\sqrt{\alpha})\textnormal{ and }N_{F/\QQ}(\alpha)\in\QQ^{\times 2}}}\cup\set{F}.\]
    The mapping from left to right is explicitly given for a pair $(\fa,u)$ as follows.  Since $\fa\in 2\Cl_F$, let $\fb\leq \cO_F$ be such that $\fa\fb^{2}=\gp{\alpha}$ is principal and $\alpha$ has square norm, after possibly swapping $\alpha$ with $-\alpha$.  The quadratic extension is $K=F(\sqrt{u\alpha})$
\end{proposition}
To understand how the discriminant condition passes through this correspondence, we have the following.

\begin{proposition}[\cite{MR3215550}*{Prop.~3.3}]\label{prop:CTProp3.3}
    Let $F$ be a number field and $K=F(\sqrt{\alpha})$ be a quadratic extension of $F$.  Write $\alpha\cO_F=\fa\fb^2$ for some squarefree ideal $\fa\leq \cO_F$ and assume that $\alpha$ is chosen so that $\fb$ is coprime to $2$.

    The relative discriminant $\fd(K/F)$ is given by the formula $\fd(K/F)=(4/\fc_\fa^2)\fa$ where $\fc=\fc_\fa$ is the largest (w.r.t. divisibility) ideal of $\cO_F$ with $\fc\mid 2\cO_F$ such that $\fc$ is coprime to $\fa$ and the congruence $x^2\equiv \alpha\mod \fc^2$ has a solution in $F$.
\end{proposition}

\subsection{Proof of \Cref{thm:fixed_cubic}}
We begin by computing the kernel in Proposition~\ref{prop:CTProp3.7}.

\begin{lemma}\label{lem:selmercount} Let $F/\QQ$ be a cyclic cubic field.  Then
\[\#\ker\braces{N_{F/\QQ}:\sel{2}(F)\to \sel{2}(\QQ)}= 4\#\Cl_F[2].\] 
\end{lemma}
\begin{proof}
Note that $\sel{2}(\QQ)=\set{\pm1}$, and moreover the involution on $\sel{2}(F)$ induced by multiplying by $-1$ swaps the image of the norm: for $\alpha\in \sel{2}(F)$ we have $N_{F/\QQ}(-\alpha)=-N_{F/\QQ}(\alpha)$.  Hence 
\[\#\ker(N_{F/\QQ}\sel{2}(F)\to \sel{2}(\QQ))=\frac{\#\sel{2}(F)}{2}.\]
We then use the short exact sequence (see e.g.~\cite{MR1728313}*{Prop.~5.2.8})
\begin{equation}
    \begin{tikzcd}
        0\arrow[r]&\cO_F^\times/\cO_F^{\times 2}\arrow[r]&\sel{2}(F)\arrow[r]&\Cl_F[2]\arrow[r]&0,
    \end{tikzcd}
\end{equation}
and Dirichlet's unit theorem, to obtain $\#\sel{2}(F)=8\#\Cl_F[2]$.
\end{proof}

We thus obtain the following.

\begin{lemma}\label{lem:fieldstoideals}
Let $F/\QQ$ be a cyclic cubic field.  Then for $X\geq 1$
    \begin{align*}
    &\#\set{L/\QQ~:~
    \begin{array}{l}
            L/\QQ\textnormal{ is an }A_4\textnormal{-quartic with}\\
        \textnormal{cubic resolvent $F$ and }\Delta_L\leq X
    \end{array}} \\
    &=
    \frac{4\#\Cl_F[2]}{3}
    \#\set{\fa\leq \cO_F~:~
    \begin{array}{l}
            \cN(\fa)\textnormal{ squarefree},\\
            \gcd(\cN(\fa),\Delta_F)=1\\
            \, [\fa]\in 2\Cl_F\leq \Cl_F, \\
            \cN(\fa)\leq \braces{\frac{X}{\Delta_F\cN(4/\fc_\fa^2)}}^{1/2}
    \end{array}}
     - \frac{1}{3}
    \end{align*}
    where $\fc_\fa$ is as defined in \Cref{prop:CTProp3.3}.
\end{lemma}
\begin{proof}  We apply Propositions \ref{prop:CTThm2.2}, \ref{prop:CTProp3.7}, and \ref{prop:CTProp3.3}, and \Cref{lem:selmercount} to obtain
    \begin{align*}
    &\#\set{L/\QQ~:~\substack{L/\QQ\textnormal{ is an }A_4\textnormal{-quartic with}\\\textnormal{cubic resolvent }F\textnormal{ and}\\\Delta_L\leq X}}  \\
    &= \frac{1}{3}\set{K/F~:~\substack{\exists\alpha\in F^\times\backslash{F^{\times2}}\textnormal{ such that}\\K=F(\sqrt{\alpha})\textnormal{, }N_{F/\QQ}(\alpha)\in\QQ^{\times 2}\\\textnormal{and }\cN(\fd(K/F))\leq X/\Delta_F}}
    \\&= \frac{4\#\Cl_F[2]}{3}\#\set{\fa\leq \cO_F~:~\substack{
    \fa\textnormal{ is squarefree with square norm,}\\ 
    [\fa]\in 2\Cl_F\textnormal{, and }
    \cN(\fa)\leq \frac{X}{\Delta_F\cN(4/\fc_\fa^2)}}
    } - \frac{1}{3}.
    \end{align*}
    We then note that for every $B\geq 1$ there is a bijection
    \begin{align*}
    \set{\fa\leq \cO_F~:~\substack{
    \fa\textnormal{ is squarefree with square norm,}\\ 
    [\fa]\in 2\Cl_F\textnormal{, and }
    \cN(\fa)\leq B}
    }
    &\leftrightarrow
    \#\set{\fa\leq \cO_F~:~\substack{
    \cN(\fa)\textnormal{ squarefree}\\
    \gcd(\cN(\fa),\Delta_F)=1\\
    [\fa]\in 2\Cl_F\leq \Cl_F\\
    \cN(\fa)\leq {B}^{1/2}
    }}\\
    \fa&\mapsto \cN(\fa)\fa^{-1},
    \end{align*}
    so the claim is immediate.
\end{proof}

By \Cref{lem:fieldstoideals} and \Cref{thm:idealcounting}, applying the bound $\cN(4/\fc_\fa^2)\leq \cN(4)=4^3$ we obtain
\begin{align*}
    &\#\set{L/\QQ~:~
    \begin{array}{l}
    L/\QQ\textnormal{ is an }A_4\textnormal{-quartic with}\\
    \textnormal{ cubic resolvent }F\textnormal{ and }\Delta_L\leq X
    \end{array} } 
    \\\geq&
    \frac{4\#\Cl_F[2]}{3}
    \#\set{\fa\leq \cO_F~:~
    \begin{array}{l}
            \cN(\fa)\textnormal{ squarefree},\\
            \gcd(\cN(\fa),\Delta_F)=1\\
            \, [\fa]\in 2\Cl_F\leq \Cl_F, \\
            \cN(\fa)\leq \braces{\frac{X}{4^3\Delta_F}}^{1/2}
    \end{array}}+O\braces{1}
    \\=&\frac{\zeta^*_F(1)L(\gamma_{F,\Delta_F}, 1)}{6\Delta_F^{1/2}}X^{1/2} + 
    O_\epsilon\braces{\#\Cl_F[2] \cdot 2^{3\omega(\Delta_F)}\left(\frac{X}{\Delta_F}\right)^{1/4+\epsilon} \log(X)^{79}}.
\end{align*}
Note that our application of \Cref{thm:idealcounting} requires the hypothesis $X\geq 64\Delta_F^{1+1/\epsilon}$ in the theorem statement, and in the main term we have used the equality  $\#\Cl_F[2] = \#\Cl_F/\#2\Cl_F$.
The stated formula in Theorem~\ref{thm:fixed_cubic} now follows from the explicit Euler product for $L(\gamma_{F,\Delta_F},1)$ shown in Lemma \ref{lem:L_gamma_F}.
\qed

\subsection{Proof of Theorem \ref{thm:main}}
We will sum the asymptotic of \Cref{thm:fixed_cubic} over cyclic cubic fields $F$ such that $3$ is unramified in $F$ and ${\Delta_F}\leq X^{\delta}$ for some small $\delta>0$.

For the error term, we use the bound $\#\Cl_F[2] \ll_\varepsilon \Delta_F^{(1+\varepsilon)/2}$ from the Brauer--Siegel Theorem \cite{MR1282723}*{Ch.~XVI} and $2^{3\omega(\Delta_F)} \ll \Delta_F^{\varepsilon/2}$ for any $\varepsilon >0$. This shows that the error becomes
\[\sum_{\substack{F\textnormal{ cyclic cubic}\\\textnormal{unram. at }3\\\abs{\Delta_F}\leq X^\delta}}O\braces{\#\Cl_F[2] \cdot 2^{3\omega(\Delta_F)}\log(\Delta_F)^{3}X^{1/4} \log(X)^{71}}=O_\delta\braces{X^{1/4+\delta\left(\epsilon + \tfrac{1}{2}\right)}\log(X)^{74}},\]

on using \eqref{eqn:number_F}.
Taking $\delta<\frac{1}{2}$, we can then choose $\varepsilon$ small enough so that $\epsilon<\frac{1-2\delta}{4\delta}$.  In particular, the sum of the error term contributes $o_\delta(X^{1/2})$, so in order to prove Theorem \ref{thm:main} it is sufficient to show that the sum of the main term is $\gg X^{1/2}\log(X)$.

The product over unramified primes in the constant $c_F$ from \Cref{thm:fixed_cubic} is absolutely convergent and uniformly bounded above and below independently of $F$.  Hence the main term is
\begin{equation} \label{eqn:gg}
    \gg X^{1/2}\sum_{\substack{F\textnormal{ cyclic cubic}\\\textnormal{unram. at }3\\\Delta_F\leq X^\delta}}\frac{\zeta_F^*(1) \prod_{p \mid \Delta_F}(1- 1/p)}{{\Delta_F}^{1/2}}.
\end{equation} 
To handle the summation, the relation $\prod_{p \mid n}(1- 1/p) = \sum_{d \mid n} \mu(d)/d$ gives
\begin{equation} \label{eqn:d_sum}
 \sum_{d \leq X^{\delta}} \frac{\mu(d)}{d} \sum_{\substack{F\textnormal{ cyclic cubic}\\\textnormal{unram. at }3\\ \Delta_F\leq X^\delta \\ d \mid \Delta_F  }}\frac{\zeta_F^*(1)}{\Delta_F^{1/2}}.
\end{equation}
We first show that we can assume that $d$ is at worst a power of log in size.

\begin{lemma}
\[ \sum_{(\log X)^{1/\delta} \leq d \leq X^{\delta}} \frac{\mu(d)}{d} \sum_{\substack{F\textnormal{ cyclic cubic}\\\textnormal{unram. at }3\\ \Delta_F\leq X^\delta \\ d \mid \Delta_F  }}\frac{\zeta_F^*(1)}{\Delta_F^{1/2}} \ll_{\delta} (\log X)^{5-1/\delta}.\]
\end{lemma}
\begin{proof}
We use the bound $\zeta_F^*(1) = O((\log \Delta_F)^{2})$ from \cite{Lou00}*{Thm.~1}. We also use that there are $2^{\omega(\Delta_F) -1}$ cyclic cubic fields of the same discriminant.  Thus the sum in the statement is
\[ \ll_\delta (\log X)^2 \sum_{(\log X)^{1/\delta} \leq d \leq X^{\delta}} \frac{1}{d} \sum_{\substack{n \leq X^{\delta/2}\\\textnormal{squarefree}\\ d \mid n }}\frac{\tau(n)}{n}\]
where $\tau$ is the divisor function. We make a change of variables and use that $\tau(dn) \leq \tau(d)\tau(n)$ to obtain
\[ \ll_\delta (\log X)^2 \sum_{(\log X)^{1/\delta} \leq d \leq X^{\delta}} \frac{\tau(d)}{d^2} \sum_{\substack{n \leq X^{\delta/2}}}\frac{\tau(n)}{n}.\]
But then we have by partial summation that
$$\sum_{\substack{n \leq X^{\delta}}}\frac{\tau(n)}{n} \ll_\delta (\log X)^2, \quad 
\sum_{(\log X)^{1/\delta} \leq d\leq X^\delta} \frac{\tau(d)}{d^2} \ll_\delta \log(X)^{1-1/\delta}.$$
Hence the claim holds.
\end{proof}

Therefore, further ensuring that $\delta<\tfrac{1}{4}$, we may assume that $d \leq (\log X)^{1/\delta}$. In which case, Theorem \ref{thm:LTAvg_d} and partial summation yields
$$\sum_{\substack{F\textnormal{ cyclic cubic}\\\textnormal{unram. at }3\\ \Delta_F\leq X^\delta \\ d \mid \Delta_F  }} \frac{\zeta_F^*(1)}{\Delta_F^{1/2}} 
=
\frac{7\pi}{26\sqrt{3}} \alpha \beta_d \cdot \frac{2^{\omega(d)}}{d}\cdot \left(1+\tfrac{1}{2}\log(X^{\delta})\right) + O_{\delta}(1).$$
Therefore combined with \eqref{eqn:gg} and \eqref{eqn:d_sum} we obtain the lower bound 
\begin{align*}
    &\#\set{L/\QQ~:~
    \begin{array}{l}
        L/\QQ\textnormal{ is an }A_4\textnormal{-quartic with } \Delta_L\leq X
    \end{array}}
    \\&\quad\quad \gg_{\delta}X^{1/2}\log(X)\sum_{d\leq \log(X)^{1/\delta}}\frac{\mu(d)2^{\omega(d)}\beta_d}{d^2}
    \end{align*}
which is sufficient for Theorem \ref{thm:main}, provided we show that the summation is non-zero. However the function $\beta_d$ is multiplicative. Hence the leading constant is proportional to
$$\sum_{d}\frac{\mu(d)2^{\omega(d)}\beta_d}{d^2} = \prod_{p \equiv 1 \bmod 3} \left(1 - \frac{2\beta_p}{p^2}\right)$$
with the sum being absolutely convergent and $d$ varying over integers which are only divisible by primes which are $1 \bmod 3$. However a short calculation shows that $0 < \beta_p < 1$ always, whence each local Euler factor is non-zero. This completes the proof of Theorem \ref{thm:main}. \qed

\section{Interpretation via stacks} \label{sec:stacks}

The counting problem in Theorem \ref{thm:fixed_cubic} fits into the recent stacky framework for Malle's conjecture \cites{MR4557890,MR4639951,DardaYasudaDMStacks,LS24}. Namely it can be interpreted as counting rational points of bounded height on the classifying stack of a Galois twist of $(\Z/2\Z)^2$. For completeness, we verify now that our count agrees with \cite{LS24}*{Conj.~9.1}, up to $2$-adic and archimedean factors.
 
We recall some of the relevant framework in our setting. Write $N = (\Z/2\Z)^2 \subset A_4$. Let $\psi:\Gal(\bar{\Q}/\Q) \to \Z/3\Z$ be a surjective homomorphism, with associated cyclic field $F_\psi$. Fixing a choice of section of $A_4 \to \Z/3\Z$, we abuse notation and view $\psi: \Gal(\bar{\Q}/\Q) \to A_4$. Let $N_\psi$ be the inner twist of $N$ with respect to $\psi$ \cite{LS24}*{Def.~2.12}. Explicitly, this is the Galois module with underlying group $N$ but such that $\Gal(\bar{\Q}/\Q)$ acts via $\psi$ permuting the $3$ non-trivial elements; it will also be convenient to view this as a group scheme below. Since $-1 \in \Q$ we have $N_\psi(-1) = N_\psi$. Let $H_\psi$ be the height function on $B N_\psi$ obtained by pulling back the discriminant height on $BA_4$. Take $\Omega$ to be the collection of non-surjective elements of $BN_\psi[k]$, as in \cite{LS24}*{Lem.~2.2}. 

\begin{lemma} \label{lem:groupoid_cardinality}
\begin{align*}
    &\#\set{L/\QQ~:~
    L/\QQ\textnormal{ is an }A_4\textnormal{-quartic with cubic resolvent }F_\psi, \abs{\Delta_L}\leq X} \\
    & = (1/3) \cdot \#\{ \varphi \in BN_\psi[k] \setminus \Omega : H_\psi(\varphi) \leq X\}.
\end{align*}
\end{lemma}
\begin{proof}
By \cite{LS24}*{Lem.~2.13}, the fibre of $BA_4 \to B(\Z/3\Z)$ over $\psi$ is $BN_\psi$. Our count is in terms of quartic fields rather than Galois fields; this requires a transition to $BS_4$ (see \cite{LS24}*{Lem.~2.15}). To keep track of the correct groupoid cardinalities one applies \cite{LS24}*{Lem.~2.14(1)} to $BN_\psi \subset BS_{4,\psi} = BS_{4}$. As $N \subset S_4$ is normal this picks up a factor of $1/|(S_{4}/N)_{\psi}(\Q)|$, where $(S_{4}/N)_{\psi}$ denotes the twist of $S_4/N$ via the conjugacy action induced by $\psi$. This can be identified with $S_3$, where $C_3$ acts via conjugacy and fixes exactly the identity element and the $3$-cycles. Hence $|(S_{4}/N)_{\psi}(\Q)| = 3$, as required.
\end{proof}

Therefore Lemma \ref{lem:groupoid_cardinality} and \cite{LS24}*{Conj.~9.1} predict that the quantity in Theorem~\ref{thm:fixed_cubic} is asymptotic to $$(1/3) \cdot c(\Q,N_\psi,H_\psi) X^{1/2}$$ where
\[
c(\Q,N_\psi,H_\psi) = |\Brun BN_\psi / \Br \Q| \cdot \tau_{H_\psi}( (\prod_v BN_\psi( (\Q_v) ) )^{\Br}).
\]
Here $\Brun BN_\psi$ denotes the unramified Brauer group of $B N_\psi$ \cite{LS24}*{Def.~5.10} and $\tau_{H_\psi}$ the Tamagawa measure associated $H$ \cite{LS24}*{\S8.4}. We first verify that the corresponding Brauer group is constant.

\begin{lemma} \label{lem:Br_constant}
    $\Brun BN_\psi = \Br \Q$.
\end{lemma}
\begin{proof}
    By \cite{LS24}*{Lem.~10.21} we have $\Brun BN_\psi / \Br \Q \cong \Sha^1_\omega(\Q,N_\psi)$.
    However as $N_\psi$ is $2$-torsion and restriction-corestriction is multiplication
    by $[F_\psi : \Q] = 3$, the map $H^1(\Q, N_\psi) \to H^1(F_{\psi}, N_{\psi})$ is injective.
    Thus it suffices to show that $\Sha^1_\omega(F_\psi,N_\psi) = 0$. But as the base-change of $N_\psi$
    to $F_\psi$ is simply $(\Z/2\Z)^2$ as a Galois module, this vanishing is elementary
    (it is a very special case of \cite{MR2392026}*{Thm.~9.1.9}).
\end{proof}

It thus remains to calculate the Tamagawa measure.

\begin{lemma}
    Let $p$ be an odd prime. Then
    $$\tau_{H_\psi,p}(BN_{\psi}(\Q_p)) =
    \begin{cases}
    p^{-v_p(\Delta_F)/2}, & \quad p \text{ ramified in } F, \\
    1 + 3/p, & \quad p \text{ split in } F, \\
    1, & \quad p \text{ inert in } F.
    \end{cases} $$
\end{lemma}
\begin{proof}
    By \cite{LS24}*{Def.~8.7} we have
    \begin{equation} \label{def:Tamagawa}
    \tau_{H_\psi,p}(BN_{\psi}(\Q_p)) = 
    \sum_{\varphi_p \in BN_{\psi}[\Q_p]} \frac{1}{|\Aut(\varphi_p)| H_p(\varphi_p)^{1/2}}.
    \end{equation}
    If $p$ is unramified in $F$ then we use the mass formula \cite{LS24}*{Cor.~8.11}, which gives
    $$\tau_{H_\psi,p}(BN_{\psi}(\Q_p)) = 1 + \frac{\# N_\psi^{\Frob_p} -1}{p} $$
    where the numerator denotes the invariants of $N_\psi$ under the action of the frobenius
    element $\Frob_p \in \Gal(F/\Q)$. This gives the stated formula.
    
    If $p$ is ramified in $F$ then we cannot apply the mass 
    formula directly since this requires $N_{\psi}$ to have good reduction at $p$. 
    Instead, 
    we note that the $\varphi_p$ being counted correspond to continuous
    homomorphisms $\Gal(\bar{\Q}_p/\Q_p) \to A_4$ such that the composition
    with $A_3 \to \Z/3\Z$ equals $\psi_p$. However since $\psi_p$ is ramified
    there is only one such homomorphism up to conjugation, namely the composition
    $\Gal(\bar{\Q}_p/\Q_p) \xrightarrow{\psi_p} \Z/3\Z \to A_4$.
    Indeed, the inertia group of $F_\psi$ at $p$ is $\Z/3\Z$.
    But the only subgroups of $A_4$ which contain $\Z/3\Z$ are $\Z/3\Z$ and $A_4$
    itself.  However $\mathrm{Im}\, \varphi_p \neq A_4$: for $p=3$ this follows from the LMFDB,
    and for $p > 3$ the ramification is tame, so
    the Galois group must be an extension of two cyclic groups. We deduce
    that $\mathrm{Im}\, \varphi_p = \Z/3\Z$, as required. 
    
    We conclude that only the identity element contributes towards \eqref{def:Tamagawa}.
    This has trivial automorphism group since automorphisms come from conjugation
    by $N_\psi(\Q_p)$, which is trivial. Therefore we obtain the measure 
    $p^{-v_p(\Delta_F)/2}$.
\end{proof}

The convergence factors come from the Artin $\L$-function $\L(\CC[N_\psi],s)/\zeta(s) = \zeta_{F_\psi}(s)$, which reads
$$\zeta_{F_\psi}(s) = \prod_{p \text{ ramified in }F}\left(1 - \frac{1}{p^s}\right)^{-1}\prod_{p \text{ split in } F}\left(1 - \frac{1}{p^s}\right)^{-3} \prod_{p \text{ inert in } F}\left(1 - \frac{1}{p^{3s}}\right)^{-1}.$$
We deduce that the Tamagawa measure equals
$$\frac{\zeta^*_{F_\psi}(1)}{\sqrt{|\Delta_F|}} \prod_{p \text{ ramified in }F}\left(1 - \frac{1}{p}\right)\prod_{p \text{ split in } F}\left(1 - \frac{1}{p}\right)^3\left(1 + \frac{3}{p}\right) \prod_{p \text{ inert in } F}\left(1 - \frac{1}{p^3}\right)$$
up to $2$-adic and archimedean factors.
Recalling Lemma \ref{lem:groupoid_cardinality}, we obtain an expression in agreement with Theorem \ref{thm:fixed_cubic} up to $2$-adic and archimedean factors, hence in agreement with \cite{LS24}*{Conj.~9.1}.

\bibliography{refs}

\end{document}